\documentclass[12pt]{amsart}
\usepackage{txfonts}
\usepackage{amscd,amssymb}
\usepackage[colorlinks,plainpages,backref,urlcolor=blue]{hyperref}
\usepackage{verbatim}
\usepackage{xcolor}

\newtheorem{theorem}{Theorem}[section]
\newtheorem{corollary}[theorem]{Corollary}
\newtheorem{lemma}[theorem]{Lemma}
\newtheorem{prop}[theorem]{Proposition}

\theoremstyle{definition}
\newtheorem{definition}[theorem]{Definition}
\newtheorem{example}[theorem]{Example}
\newtheorem{remark}[theorem]{Remark}

\newcommand{\N}{\mathbb{N}}
\newcommand{\Z}{\mathbb{Z}}

\newcommand{\C}{\mathbb{C}}

\newcommand{\PP}{\mathbb{P}}

\renewcommand{\k}{\Bbbk}

\DeclareMathAlphabet{\pazocal}{OMS}{zplm}{m}{n}
\newcommand{\A}{{\pazocal{A}}}
\newcommand{\B}{{\pazocal{B}}}

\newcommand{\OO}{{\pazocal O}}

\renewcommand{\k}{\Bbbk}

\newcommand{\F}{{\mathcal{F}}}

\newcommand{\cE}{{\mathcal{E}}}

\DeclareMathOperator{\im}{Im}
\DeclareMathOperator{\Coker}{Coker}
\DeclareMathOperator{\Ker}{Ker}

\DeclareMathOperator{\Der}{{Der}}

\def\dot{\mathchar"013A}
\newcommand{\hdot}{{\raise1pt\hbox to0.35em{\Huge $\dot$}}}

\begin{document}
\date{July, 2022}

\title[On some freeness-type properties for line arrangements]%
{On some freeness-type properties for line arrangements}

\author[T. Abe]{Takuro~Abe}
\address{Institute of Mathematics for Industry, Kyushu University, Fukuoka 819-0395, Japan.}
\email{abe@imi.kyushu-u.ac.jp}

\author[D. Ibadula]{Denis~Ibadula}
\address{Ovidius University, Faculty of Mathematics and Informatics, 124 Mamaia Blvd., 900527 Constan\c{t}a, Romania}
\email{Denis.Ibadula@univ-ovidius.ro}

\author[A. M\u acinic]{Anca~M\u acinic}
\address{Simion Stoilow Institute of Mathematics, 
 Bucharest, Romania}
\email{Anca.Macinic@imar.ro}



\begin{abstract} 
By way of Ziegler restrictions we study the relation between nearly free plane arrangements and combinatorics and we give a Yoshinaga-type criterion  for plus-one generated plane arrangements. 
\end{abstract}
 
\maketitle

\section{Introduction} 
\label{sec:introduction}

Let $\k$ be a field and $V$ a $3$-dimensional vector space over $\k$. Choose a basis $\{x,y,z\}$ for the dual vector space $V^*$ so as to describe the coordinate ring $S:=\k[x,y,z]$.
 Let $\A$ be a non-empty central arrangement of planes in $V$, see Section \S \ref{sec:preliminaries} for details.
  We will identify $\A$, whenever convenient, to the canonically associated arrangement of lines in the corresponding projective plane $\PP^2= {\rm Proj} (S)$. Let $\alpha_H \in V^*$ be such that $H \in \A$ is defined by $\alpha_H =0$. 
We are mainly interested here in arrangements of planes. Consequently, we state only in this particular context, i.e. only for arrangements of planes, some notions and results that hold also in higher dimensions. We refer to \cite{OT} for a comprehensive presentation of the theory of arrangements of hyperplanes in general.
 
Define 
$$
D(\A) = \{ \theta \in \Der S \; | \;   \theta( \alpha_H) \in S \alpha_H, \; \forall H \in \A\}
$$
 
\noindent the logarithmic derivation module of the arrangement $\A$, a graded submodule of the $S$-module $\Der S = S \partial_x \oplus S \partial_y \oplus S \partial_z$ of derivations of $S$. $D(\A)$ is not generally a free $S$-module. When it is free, the arrangement $\A$
is called {\it free} and the ordered sequence of degrees of a basis of $D(\A)$ constitutes the {\it exponents} of $\A$, denoted $\exp(\A)$. Since the Euler derivation, $\theta_E = x \partial_x+ y  \partial_y + z \partial_z $, is always an element in $D(\A)$, the sequence of exponents always starts with $1$. We will omit from now on, unless explicitly stated, the first exponent. So, by a free arrangement of exponents $(a,b)$ one actually means a free arrangement of exponents $(1,a,b)_{\leq}$.

We need to recall the definition of a notion that turned out to be important in the study of freeness, {\it multiarrangements}. Multiarrangements are, simply put, arrangements with each hyperplane labelled by a positive integer. More precisely, 
a multiarrangement is a pair consisting of an arrangement $\A$ and a map $m: \A \rightarrow \Z_{\geq 0}$, called {\it multiplicity}. 
The Ziegler restrictions defined in \ref{def:Ziegler_restr} are the examples of multiarrangements that we will work with.

For $H \in \A$, let 
$$\A^H = \{ H \cap K\; | \; K \in \A \setminus \{ H \} \}$$

\noindent be the arrangement induced by $\A$ on $H$. One has a multiarrangement structure on $\A^H$, called the {\it Ziegler restriction} (\cite{Z}), defined as follows.

\begin{definition}
\label{def:Ziegler_restr}
Let $m^H : \A^H \rightarrow \N_{>0}$ be defined by the correspondence:
$$
X \mapsto \# \{ K \in \A \; |\; X \subset K \} - 1.
$$
We call the pair $(\A^H, m^H)$ the Ziegler restriction of $\A$ onto $H$.  
\end{definition}

The derivation module associated to the multiarrangement $(\A^H, m^H)$ is defined by 
$$
D(\A^H,m^H) = \{ \theta \in Der \overline{S}\; |\; \ \theta( \alpha_K) \in \overline{S} (\alpha_K)^{m^H(K)}, \; \forall K \in \A^H \},
$$

\noindent where $\alpha_K \in \overline{S}$ is the linear form that defines $K$ and $\overline{S} = S/(\alpha_H)$. 

If $\A$ is an arrangement in a $3$-dimensional vector space, then $D(\A^H, m^H)$ is always free, of rank $2$.  For $\{ \theta_1^H, \theta_2^H \}$ homogeneous basis in $D(\A^H, m^H)$, we call  the set $(\deg(\theta_1^H), \deg(\theta_2^H))$ the exponents of $(\A^H, m^H)$, denoted by $\exp(\A^H, m^H)$.

Working with Ziegler restrictions plays an important role in proving the main results of this paper. \\

Let $H \in \A$ be defined by the linear form $\alpha_H$ and 
$$
D_H(\A) = \{ \theta \in D(\A) \; | \; \theta(\alpha_H) = 0\}.
$$
Define the {\it Ziegler restriction map}:
$$ 
  \pi: D_H(\A) \rightarrow D(\A^H, m^H) 
$$
 by taking \text{\text{modulo}} $\alpha_H$, see \cite{Z} for details. 
 
 Hence we have the exact sequence:
  $$
   0 \rightarrow D_H(\A)  \overset{\cdot \alpha_H}\rightarrow D_H(\A) \overset{\pi} \rightarrow D(\A^H, m^H) 
 $$

We recall a couple of important results due to Ziegler and  Yoshinaga that link freeness and Ziegler restrictions. 

\begin{theorem}
\label{thm:Ziegler}(Ziegler, \cite{Z})
Let $\A$ be free of exponents $(a,b)$. Then, for any $H \in \A$, the Ziegler restriction $(\A^H, m^H)$ is also free of exponents $(a,b)$ and the Ziegler restriction map is surjective.
\end{theorem}

The converse result was proved by Yoshinaga.
Let 
$$\chi_0(\A, t) = \chi (\A, t) / (t-1),$$
where $\chi(\A, t)$ is the characteristic polynomial of the arrangement $\A$, see \cite{OT}. Then  
$$\chi_0(\A, t) = t^2 - (|\A| - 1)t + b_2^0(\A),$$
and, when $\k = \C$, we have that $b_2^0(\A) = b_2(\A) - (|\A|-1)$, where $b_2(\A)$ is the second Betti number of the complement of the arrangement.

\begin{theorem}(Yoshinaga's criterion, \cite{Y0})
\label{thm:Yoshinaga}
Let $(\A^H, m^H)$ be the Ziegler restriction  of $\A$ onto $H \in \A$ with $\exp(\A^H, m^H) = (e_1, e_2)$. Then $b_2^0(\A) - e_1e_2 = \dim_{\k}\Coker(\pi) < \infty$ and $b_2^0(\A) - e_1e_2 = 0$ if and only if $\A$ is free with $\exp(\A) = (e_1,e_2)$.
\end{theorem}

\noindent The freeness property for arrangements is intensively studied from algebraic, geometric and combinatorial perspectives. A long standing conjecture in the field of hyperplane arrangements, concerning free arrangements, is the {\bf Terao conjecture}, which states that whether an arrangement $\A$ is free or not depends only on its combinatorics.\\

We turn our attention in this paper to some recently introduced (weaker) notions of freeness: {\it nearly free} arrangements of lines in the projective plane, that were defined by Dimca and Sticlaru in \cite{DS1}, and the {\it plus-one generated} arrangements introduced by the first author in \cite{A0}, see Section \S \ref{sec:preliminaries} for the definitions.\\

In Section \S \ref{sec:nfree}, we mostly look at nearly free arrangements through combinatorial perspective. We prove that, under certain restrictions, freeness and near freeness depend on the combinatorics.

\begin{theorem}
\label{thm:dist_exp}
Let $\A$ be a plane arrangement in $\C^3$. Let $\chi_0(\A, t) = (t-n)(t-n-r)+1$ with $n, r \geq 0$.  If $r \geq n > 0$, then whether $\A$ is free or nearly free depends only on $L(\A)$.
\end{theorem}

Consequently, if there is enough distance between the exponents, a nearly free version of the Terao conjecture holds. We apply this in Corollary \ref{cor:d/3} to show that if $\A$ is a nearly free arrangement of exponents $(a,b)_{\leq}$ (see Definition \ref{def:POG}) and $a \leq \frac{|\A|-1}{3}$, then all other arrangements in the lattice isomorphism class of $\A$ are nearly free.

In Theorem \ref{thm:d_1leq5} we show that a nearly free version of the Terao conjecture holds when limiting the values of the exponents. More precisely we prove that if $\A$ is a nearly free arrangement of exponents $(a,b)$ such that at least one of the exponents is most $5$, then  all other arrangements in the lattice isomorphism class of $\A$ are nearly free.
\\

A Yoshinaga-type criterion for near freeness is given by the first author and Dimca in \cite{AD}.

\begin{theorem}(\cite{AD})
\label{thm:nfree_criterion}
Let $\chi_0(\A, t) = t^2 - (|\A| - 1)t + b_2^0(\A)$. Then $\A$ is nearly free if and only if there exists $H \subset \PP^2$ such that $b_2^0(\A) - e_1e_2 = 1$, where $(e_1, e_2)$ is the splitting type onto $H$ for the bundle of logarithmic vector fields associated to $\A$.
\end{theorem}

On the same note, we give in Theorem \ref{thm:POG_characterisation} a similar criterion for an arrangement to be plus-one generated. To state it we need to add a bit to the set-up.\\

\noindent Consider the following property: \\
 
{ \bf[P]}  Let $\A$ be an arrangement for which there exists $H \in \A$ such that one of the conditions below holds.
 \begin{enumerate}
 \item There exists  $ \alpha \in S$ non-zero linear form such that $\langle \theta_1^H, \theta_2^H \rangle \supsetneq Im(\pi) \supset \langle \theta_1^H, \alpha \theta_2^H \rangle$
 \item There exists  $ \beta \in S$ non-zero linear form such that $\langle \theta_1^H, \theta_2^H \rangle \supsetneq Im(\pi) \supset \langle \theta_2^H, \beta \theta_1^H \rangle$
 \end{enumerate}
 
\noindent where $\{\theta_1^H, \theta_2^H\}$ is a basis in $D(\A^H, m^H) , \; 
 \pi$ is the Ziegler restriction map and $\langle X \rangle$ denotes the $S$-module generated by a set $X$.

\begin{theorem}
\label{thm:POG_characterisation}
 $\A$ satisfies property {\bf[P]} if and only if it is plus-one generated.
\end{theorem}
 
The last section is generally devoted to results on plus-one generated arrangements. We show in Theorem \ref{thm:splitting_type} that the possible splitting types onto any projective line for the bundle of logarithmic vector fields associated to the arrangement are determined by the exponents.
 See  \cite{AD, MV} for similar results on nearly free arrangements: in \cite{AD} the authors describe the possible splitting types for the vector field associated to nearly free curves, and the result is recovered by Marchesi-Vall\`es in \cite{MV}, who moreover give a geometric characterization of the set of jumping lines. 
 In Corollary \ref{cor:exp} we prove that the exponents of a plus-one generated arrangement put strong restrictions on the exponents of its Ziegler restrictions. Along the way we generalize some other results known to hold for nearly free arrangements.
 
 \section*{Acknowledgements}
The first author is partially supported by JSPS KAKENHI
 Grant-in-Aid for Scientific Research (B) Grant Number JP16H03924. The third author is  partially supported by 
 a grant of the Romanian Ministry of Education and Research, CNCS - UEFISCDI, project number PN-III-P4-ID-PCE-2020-2798, within PNCDI III.

\section{Preliminaries}
\label{sec:preliminaries}

We will present in more detail some of the notions we will work with in the next chapters.

An arrangement of hyperplanes $\A$ in a vector space $V$ over a field $\k$ is a finite collection of hyperplanes in $V$. $\A$ is called {\it central} if $\cap _{H \in \A} H \neq \emptyset$. 
 Recall that we work with $\A \subset V = \k^3$ a central non-empty arrangement and $S:=\k[x,y,z]$. We present here the definition of plus-one generated arrangements in $V = \k^3$ (see \cite{A0} for the general definition).

\begin{definition}
\label{def:POG}
The arrangement $\A$ is plus-one generated with exponents $(a,b)$ and level $d$ if there is a minimal set of homogeneous generators $\{ \theta_1 = \theta_E, \theta_2 ,\theta_3, \phi \}$  for $D(\A)$ with $\deg(\theta_2) = a, \; \deg(\theta_3) = b, \deg(\phi) = d$, and a unique homogeneous relation
$$
f_1 \theta_E+ f_2 \theta_2 + f_3 \theta_3 + \alpha \phi = 0
$$
where $f_1, f_2, f_3, \alpha \in S, \; \deg(\alpha) = 1$.
\end{definition}

If moreover $b=d$ in the previous definition, one gets a {\it nearly free} arrangement. \\

It is well known that, provided that $\A \neq \emptyset$, there is an isomorphism
$$
D(\A) \cong S \theta_E \oplus D_H(\A) 
$$

\noindent and as a consequence the relation among the generators of the derivation module from the definition of the plus-one generated arrangement can be assumed to take place in $D_H(\A)$.

Plus-one generated arrangements are defined by taking one step further from the algebraic definition of free arrangements. As it turns out, they are quite literally next to free arrangements, as any free arrangement produces by deletion or addition either a free or a plus-one generated arrangement, see \cite{A0}. So, results that hold for free arrangements provide legitimate questions to explore for plus-one generated arrangements, and in particular for nearly free ones, as we shall see in the next chapters.

One of them concerns the splitting type of the bundle of logarithmic vector fields associated to the arrangement. The set-up is the following.

Say $\A \subset \C^3$ is defined as the zero set of the homogeneous polynomial $f^{\A} \in S = \C[x,y,z]$. Let $f^{\A}_x, f^{\A}_y, f^{\A}_z$ be the partial derivatives of $f^{\A}$ with respect to $x,y,z$. Consider the graded S-module of Jacobian syzygyes of $f^{\A}, \; AR(f^{\A})$, defined by 
$$
AR(f^{\A}) = \{ (a,b,c) \in S^3\; | \; a f^{\A}_x + b  f^{\A}_y + c f^{\A}_z = 0 \},
$$
\noindent connected by a natural isomorphism
$$
(a,b,c) \mapsto a \partial_x + b \partial_y + c \partial_z
$$
to the module of logarithmic derivations of the arrangement $\A$ that send $f^{\A}$ to zero, 
$$
D_0(\A) = \{ \theta \in \Der S \; | \; \theta(f^{\A}) = 0 \}.
$$
Moreover, there is a well known isomorphism 
$$
D_0(\A) \cong D_H(\A),
$$
 for any $H \in \A$, defined by 
$$\theta \mapsto \theta - (\theta (\alpha_H)/ \alpha_H)\theta_E,
$$
 see for instance  \cite[Proposition 2.10]{AD} for details. 
 
Then the freeness of $\A$ is immediately equivalent to the freeness of the S-module $AR(f^{\A})$.\\

Let $\widetilde{AR(f^{\A})}$ be the associated sheaf of $AR(f^{\A}) $ over $\PP^2$.
Consider the graded S-module of sections associated to a sheaf $\F$ on $\PP^2,$
$$
\Gamma_*(\F) = \bigoplus_{k \in \Z} \Gamma(\PP^2, \F(k)).
$$
$\widetilde{AR(f^{\A})}$ is a locally free sheaf, i.e. a vector bundle on $\PP^2$. By \cite{AY}, the natural homomorphism $AR(f^{\A}) \rightarrow \Gamma_*(\widetilde{AR(f^{\A})})$, that sends $x \mapsto \frac{x}{1}$,
 is in fact an isomorphism, hence we can assume the identification $\Gamma_*( \widetilde{AR(f^{\A})}) = AR(f^{\A})$.

If $\cE$ is a rank $2$ vector bundle on $\PP^2$ and $L \subset \PP^2$ a line, we know that the restriction of $\cE$ to $L, \; \cE |_{L}$, which is a vector bundle on $L\equiv \PP^1$, always splits (by Grothendieck's splitting theorem), say $\cE |_{L} = \OO_L(-e_1) \oplus  \OO_L(-e_2)$. The pair of (non-negative) integers $(e_1, e_2)$ is called {\it the splitting type} of $\cE$ along the line $L$. 
 
Finally, let us recall a result of Yoshinaga on the splitting type of a vector bundle of rank $2$ on $\PP^2$ along a line. 

\begin{theorem}(\cite{Y})
\label{thm:pi_L}
Let $\cE$ be a rank $2$ vector bundle on $\PP^2$ and $L \subset \PP^2$ a line. Let $\cE|_{L} =\OO_L(-e_1) \oplus  \OO_L(-e_2)$. Then $c_2(\cE) \geq e_1 e_2$, furthermore 
$$
c_2(\cE) - e_1 e_2 = dim_{\k}Coker(\Gamma_*(\cE) \overset{\pi_L}{\longrightarrow}(\Gamma_*(\cE|_{L}))
$$
where $c_2(\cE)$ is the second Chern number of $\cE$ and $\pi_L$ is the morphism of graded modules induced by the restriction to $L$ of the vector bundle $\cE$. Moreover, $\cE$ is splitting if and only if $c_2(\cE) = e_1 e_2$.
\end{theorem}

A remarkable fact is that for the vector bundle $\widetilde{AR(f_{\A})}$, to determine the splitting type onto the lines in $\A$ means to determine the exponents of the corresponding Ziegler restrictions, see  \cite{Y0}. Refer to
\cite{AD} for a comprehensive analysis on the splitting types of bundles of logarithmic vector fields associated to reduced plane curves in general. 
In particular, notice that the above result implies Theorem \ref{thm:Yoshinaga}.

The splitting type onto arbitrary projective lines, in particular for the lines in $\A$, for the vector bundle $\widetilde{AR(f_{\A})}$ when $\A$ is a nearly free arrangement, is described in \cite[Corollary 3.4, Theorem 5.3]{AD}, see also \cite[Proposition 2.4]{MV}. 
Our Theorem \ref{thm:splitting_type}  and Corollary  \ref{cor:exp} extend these results to plus-one generated arrangements.

\begin{theorem}
\label{thm:thmAD5.3} (\cite{AD})
Let $\A$  be nearly free with exponents $(a, b)$. Then, for any $H \in \A, \; \exp(\A^H, m^H) = (a-1, b)$ or $(b-1, a)$. 
\end{theorem}

Let us now turn our attention to the combinatorics of line arrangements and recall some basic notions and results that we will use throughout the next sections. 
\noindent For an arrangement $\A$, let $L(\A) = \{ \cap_{H \in \mathcal{B}}H \; | \; \mathcal{B} \subset \A \}$ be the {\it intersection lattice} of $\A$, endowed with a poset structure with partial ordering given by reverse inclusion. An invariant associated to $\A$ or a property of $\A$ is {\it combinatorial}, or {\it combinatorially determined}, if it depends solely on the intersection lattice of $\A$.\\

For any $X \in L(\A)$, let $\A_X = \{H \in \A \; |\; X \subseteq H\}$. If $X \in L(\A)$ is of codimension $2$, then set $m(X) = |\A_X|$. When $\A$ is identified to an arrangement of lines in $\PP^2$, then $X$ is just a point in the projective plane, and we call $m(X)$  the {\it multiplicity} of the point $X$.  Denote by $m(\A)$ the maximal multiplicity over all multiple points of $\A$.\\

For $H \in \A$, denote $n_H := |\A^H|$. Notice that if one looks at $\A$ as an arrangement of lines in $\PP^2$, then $n_H$ is just the number of intersection points on the line $H$.
We already know that if $\A$ is a free arrangement, there are restrictions on the values of $n_H$ in terms of the exponents of $\A$.

\begin{theorem} (\cite{A2})
\label{thm:cor1.2A2}
Let $\A$ be a free arrangement with $\chi_0(\A, t) = (t-n)(t-n-r)$ with $n, r \geq 0$. Then:
\begin{enumerate}
\item If $H \in \A$ then $n_H \leq n+1$ or $n_H = n+r+1$.
\item If $L \notin \A$, let  $\B = \A \cup \{L\}$. Then $|\B^L| = n+1$ or $|\B^L| \geq n+r+1$.
\end{enumerate}
\end{theorem}

 We will see in the next section that similar types of weak combinatorics restrictions are imposed by near freeness.
In \cite{AD} sufficient conditions for near freeness are formulated in terms of weak combinatorics:
 
 \begin{theorem}(\cite{AD})
 \label{thm:5.8AD}
 Let $\A$ be an arrangement with $\chi_0(\A, t) = (t-a)(t-b)+1$ with $a, b$ real numbers and $a \leq b$. Then $\A$ is nearly free if there is $H \in \A$ such that one of the following holds:
 \begin{enumerate}
 \item $n_H = b+1$
 \item$n_H = a+1$ and $b \neq a+2$.
 \end{enumerate}
 \end{theorem}
 
 Conditions on the weak combinatorics of an arrangement can determine the exponents of the Ziegler restrictions. This leads to partial positive answers for the Terao and near Terao conjectures.

\begin{prop}(\cite{YW})
\label{prop:Yosh_comb}
Let $\A$ be an arrangement and let $(\A^H, m^H)$ be the Ziegler restriction of $\A$ onto $H$, for some $H \in \A$. 
\begin{enumerate}
\item If $n_H \geq (|\A|+1)/2$ then $\exp(\A^H, m^H) = (|\A|-n_H, n_H-1)$.
\item If there is $X \in \A^H$ with $m := m^H(X)  \geq (|\A|-1)/2$ then $\exp(\A^H, m^H) = (m, |\A|-1-m)$.
\end{enumerate}
\end{prop}

It is immediate that in the hypothesis from Proposition \ref{prop:Yosh_comb}(2), Theorem \ref{thm:Yoshinaga} implies at once that the Terao conjecture holds. It turns out that under the same hypothesis, the near Terao conjecture holds as well.

\begin{prop}(\cite{AD})
\label{prop:nearly_non_balanced}
Let $\A$ be an arrangement such that for some $H \in \A$ there is $X \in \A^H$ with $m^H(X)  \geq (|\A|-1)/2$. Then the near freeness of $\A$ depends only on $L(\A)$. 
\end{prop}

A multiplicity $m: \A \rightarrow \Z_{\geq 0}$ is called {\it balanced} if $ m(H)  \leq (|\A|-1)/2$ for all $H \in \A$, and {\it non-balanced} otherwise. An arrangement is called {\it balanced} if all its Ziegler restrictions are balanced and {\it non-balanced} otherwise.
Thus Proposition \ref{prop:nearly_non_balanced} states that near freeness depends only on the intersection lattice for non-balanced arrangements and Proposition \ref{prop:Yosh_comb}(2) implies the same for freeness, via the computation of the exponents for a Ziegler restriction. 

\noindent One has an instrument to handle the exponents of the Ziegler restrictions, even in the case of balanced arrangements:

\begin{theorem}(\cite{A01})
\label{thm: exp_dist_A01}
Let $\A$ be an arrangement in a $2$-dimensional vector space over a field of characteristic zero with $|\A|>2$. If $m:\A \rightarrow \Z_{>0}$ is balanced and $\exp(\A,m)=(e_{1},e_{2})$, then $|e_1-e_2| \leq |\A|-2$.
\end{theorem}

For future references, we need to recall a result on the behaviour of near freeness to addition-deletion techniques.

\begin{theorem}(\cite{AD})
\label{thm:5.10AD}
Let $\A$ be an arrangement, $H \in \A$ and $\mathcal{B} := \A \setminus \{H\}$. Let $a \leq b$ be two non-negative integers. Then, any two of the following imply the third:
\begin{enumerate}
\item $\A$ is nearly free with exponents $(a+1, b+1)$.
\item $\mathcal{B}$ is free with exponents $(a, b)$.
\item $|\A^H| = b+2$.
\end{enumerate}
\end{theorem}

\section{Nearly free line arrangements and combinatorics} 
\label{sec:nfree}

Throughout this section, unless otherwise stated, we will assume that $\A$ is a central arrangement in  $\C^3$.

\begin{prop}
\label{prop:n_H_restr}
Let $\chi_0(\A, t) = (t-n)(t-n-r)+1$ with $n, r \geq 0$. If $r\neq 2$, then there are no $H \in \A$ such that $n + 1 < n_H < n + r + 1$. If $r=2$, then the same conclusion holds, except when $\A$ is free with $\exp(\A)=(n+1,n+1)$. In this case, if $n_H> n + 1$, for some $H \in \A$ then, $n_H = n + 2$.
\end{prop}

\begin{proof}
The statement is clear if $r < 2$. Assume that $r \geq 2$ and there exists $H \in \A$ such that $n+1<n_H<n+r+1$. Let  $(\A^H, m^H)$ be the Ziegler restriction of $\A$ onto $H$ and $(e_1,e_2)_{\leq}$ its exponents. 

We will show that $n+1 \leq e_1 \leq e_2 \leq n+r-1$.  Notice that  $e_1 + e_2 = 2n +r$, so it is enough to prove that $e_1 \geq n+1$. There are two cases to be considered. 
If  $n_H \leq (|\A|+1)/2$, then $e_1 \geq n_H - 1$ follows incrementally from 
 \cite[Lemma 4.3]{AN}.
If $n_H >  (|\A|+1)/2$, then, by Proposition \ref{prop:Yosh_comb}, $e_1 \in \{ n_H - 1, |\A| - n_H \}$. By the inequality $n+1<n_H<n+r+1$ both $n_H - 1 > n$ and  $|\A| - n_H > n$, and this settles the second case.

Since $e_1 + e_2 = 2n +r$ and $n+1 \leq e_1 \leq e_2 \leq n+r-1$, it holds that $e_1e_2 \geq  (n + 1)(n + r - 1)$. 

Compute

\begin{center}
$0 \leq b_2^0(\A) - e_1e_2 \leq n(n + r) + 1 - (n + 1)(n + r - 1)= 2 - r\leq 0,$
\end{center}

which contradicts Yoshinaga's criterion (Theorem \ref{thm:Yoshinaga}), if $r\neq 2$.

If $r=2$, then $b_2^0(\A) - e_1e_2=0$ and thus, by Yoshinaga's criterion, $\A$ is free with $\exp(\A)=(n+1,n+1)$. In this case, if $n_H > n + 1$, then, by Theorem \ref{thm:cor1.2A2}, $n_H = n + 2$. 
\end{proof}

\begin{prop}
\label{prop:exp_dist}
Assume that $\A$ is balanced, and there is $H \in \A$ such that $r=n_H - 2$ and $\chi_0(\A, t) = (t-n)(t-n-r)+1$ with $n, r \geq 0$. If $r \neq 2$ then $\A$ is nearly free. If $r=2$, then $\A$ is either nearly free or free with exponents $(n+1, n+1)$.
\end{prop}

\begin{proof}
 Let $(\A^H, m^H)$ be the Ziegler restriction of $\A$ onto $H$, and $(e_1, e_2)_{\leq}$ its exponents. $\A$ is balanced and thus, by Theorem \ref{thm: exp_dist_A01}, $e_2 - e_1 \leq n_H - 2 = r$. Since $e_1 + e_2 = |\A| -1 = 2n + r$, it holds that $e_1e_2 \geq n(n + r)$. If $(e_1, e_2) = (n, n + r)$, then $b_2^0(\A) - e_1e_2 = 1$, hence nearly free, by Theorem \ref{thm:nfree_criterion}. Assume not, then $n < e_1 \leq e_2 < n + r$ and, consequently, $r\geq 2$. Thus
$$0 \leq b_2^0(\A) - e_1e_2 \leq n(n + r) + 1 - (n + 1)(n + r - 1) = 2 - r\leq 0.$$
Hence, $r=2$ and $b_2^0(\A) - e_1e_2 =0$. Consequently, $\A$ is free with exponents $\exp(\A)= (n+1, n+1)$.
\end{proof}

\subsection{Proof of Theorem \ref{thm:dist_exp}} 
As recalled, if $\A$ is not balanced, then freeness and near freeness depend only on $L(\A)$. So assume that $\A$ is balanced and $H \in \A$ arbitrary.

Moreover, we may assume that $r \neq 2$, since $r=2$ implies $n \leq 2$, hence $|\A| = 2n+r+1 \leq 7$  and in this range both freeness and near freeness are combinatorial, see  \cite{ACKN} and Theorem \ref{thm:nfree_12lines}, so the conclusion of Theorem \ref{thm:dist_exp} holds in this case.

 By Proposition \ref{prop:n_H_restr}, $n_H \leq n + 1$ or $n_H \geq n + r + 1$. $\A$ is nearly free if either equality holds, see Theorem \ref{thm:5.8AD}. 

 So we may further assume that $n_H \leq n \leq r$ or $n_H \geq n + r + 2$, for any $H \in \A$. 

Assume first that there is $H \in \A$ such that $n_H \leq n \leq r$. Let $\exp(\A^H, m^H) = (e_1, e_2)_{\leq}$. By Theorem \ref{thm: exp_dist_A01}, $e_2 - e_1 \leq n_H - 2 \leq r - 2$. 
Notice that this already implies, since $r \neq 2$, that in fact $r>2$.
Since $e_1 + e_2 =2n + r$, it holds that $e_1e_2 \geq (n + 1)(n + r - 1)$. Now
$$
0 \leq b_2^0(\A) - e_1e_2 \leq n(n + r) + 1 - (n + 1)(n + r - 1) = 2 - r,
$$

\noindent which is only possible if $r=1$, contradiction to $r>2$. It follows that the case  $n_H \leq n \leq r$, for some $H \in \A$, cannot happen.

Then we are left with the case $n_H \geq n + r + 2$ for all $H \in \A$. By Theorem \ref{thm:cor1.2A2} and \cite[Proposition 4.6]{DIM_arx}, $\A$ is neither free nor nearly free if $n_H \geq n + r + 3$, for some $H \in \A$. So we may assume that $n_H = n + r + 2$ for all $H \in \A$. Recall that $b_2^0(\A) = n(n + r) + 1$. Fix an arbitrary $H \in \A$ and set $\A' = \A \setminus \{H\}$. Then $b_2^0(\A') = n(n + r) + 1 - (n + r + 1) = (n - 1)(n + r)$. Let $L \in \A'$ and assume that $\A_{H \cap L} =  \{H, L\}$. Then $|(\A')^L|= n+r+1$ and, by Proposition \ref{prop:Yosh_comb}(1),
 $\exp ((\A')^L, m^L) = (n-1, n+r)$. Hence Yoshinaga’s criterion (Theorem \ref{thm:Yoshinaga}) shows that $\A'$ is free with exponents $(n - 1, n + r)$.  Since $|\A^H| = n + r + 2$, by Theorem \ref{thm:5.10AD}, $\A$ is nearly free with exponents $(n, n + r+1)$.  Finally, notice that since $|\A| = 2n + r + 1$ and $|\A^H| = n + r + 2$, there exists always a line $L \in \A$ such that $\A_{H \cap L} = \{H, L\}$. 
\qed

\begin{corollary}
\label{cor:d/3}
Let $\A$ be a nearly free arrangement of exponents $(a, b)_{\leq}$ such that $a \leq \frac{|\A|-1}{3}$. Then all other arrangements in the lattice isomorphism class of $\A$ are again nearly free.
\end{corollary}

\begin{proof}
Immediate from Theorem \ref{thm:dist_exp} and its proof. 
\end{proof}

A result mirroring the previous one holds for free arrangements, and it is an immediate consequence of the following result.

\begin{theorem}(\cite{A2})
\label{thm:cor5.4A2}
Let $\A$ be an arrangement with $\chi_0(\A, t) = (t-n)(t-n-r)$ with $n, r \geq 0$. If $r \geq n-3$ then the freeness of $\A$ depends only on $L(\A)$.
\end{theorem} 

Another consequence of Theorem \ref{thm:cor5.4A2} is the following.

\begin{corollary}(\cite{A2})
\label{cor:cor5.5A2}
Let $\A$ be an arrangement with $\chi_0(\A, t) = (t-n)(t-n-r)$ with $n, r \geq 0$. If $\{n,n+r\} \cap \{0,1,2,3,4,5\} \neq \emptyset$, then the freeness of $\A$ depends only on $L(\A)$.
\end{corollary}

We list some more restrictions to the weak combinatorics of a line arrangement due to near freeness, see Corollary \ref{cor:maxline} for a similar result on plus-one generated arrangements.

\begin{lemma}
\label{lemma:generic_nfree}
A generic arrangement $\A$ is nearly free if and only if $|\A| = 4$.
\end{lemma}

\begin{proof}
It follows immediately from \cite[Corollaries 3.5, 3.6 and 4.12]{DIM0}.
\end{proof}

\begin{prop}
\label{prop:n_H_max}
Let $\A$ be a nearly free arrangement of exponents $(a,b)_{\leq}$. Then  $n_H \leq  b+1$, for all $H \in \A$. Moreover, if $\A$ is different from the generic arrangement of cardinal $4$ then there exists $H \in \A$ such that $n_H < b+1$.
\end{prop}

\begin{proof}
The first part of the claim,  the fact that $n_H \leq  b+1$, for all $H \in \A$,  is proved in \cite[Proposition 4.6]{DIM_arx}. We include the proof here for completion.
By Theorem \ref{thm:thmAD5.3}, for any $H \in \A$, the exponents of the Ziegler restriction of $\A$ onto $H$ are either $(a-1, b)$ or $(a, b-1)$. Assume that there is $H \in \A$ such that $n_H > b+1$. Then $\exp(\A^H, m^H)=(|\A|-n_H, n_H-1)$, by Proposition \ref{prop:Yosh_comb}(1). Consequently either $|\A|-n_H = a-1$ or $|\A|-n_H = a$, that is, $n_H  \in \{b, b+1\}$, contradiction to the assumption $n_H > b+1$.

Now, for the second part of the claim, assume $\A$ is different from the generic arrangement on $4$ lines. Assume that there exists a nearly free arrangement $\A$ of exponents $(a,b)_{\leq}$ with the property that,  for all $H \in \A, \; n_H = b+1$.
Then, by Theorem \ref{thm:5.10AD}, any subarangement $\B = \A \setminus \left\{H\right\}$ is free with exponents $(a-1, b-1)$. Then, for this free arrangement $\B$, by Theorem \ref{thm:cor1.2A2}, it follows that $|\B^K|$ should be at most $(b-1)+1 = b$, for any $K \in \B$. 

Notice that we may safely assume that $\A$ is not generic, since $\A$ is nearly free and the only generic nearly free arrangement is  the generic arrangement on $4$ lines, by Lemma \ref{lemma:generic_nfree}.

Hence, one may just choose a line $H$ that has a point of multiplicity greater than $2$. Then in the arrangement  $\B = \A \setminus \left\{H\right\}$ one of the remaining lines still contains $b+1$ multiple points.  So, the resulting free arrangement arrangement $\B$ would contain a line with $b+1$ multiple points, contradicting the above.
\end{proof}

\begin{corollary}
\label{cor:dbld_2+1}
For a nearly free arrangement $\A$ with exponents $(a,b)_{\leq}$ any two lines $H,K \in \A$ with $n_H = n_K = b + 1$ should intersect at a double point.
\end{corollary}

\begin{proof}
It is immediate from the proof of Proposition \ref{prop:n_H_max}.
\end{proof}

\begin{prop}(\cite{DIM_arx})
\label{prop:14nfree}
Let $\A$ be a nearly free arrangement with $|\A| \leq 14$ and exponents $(a,b)_{\leq}$. If there exists a line in $\A$ with precisely $b+1$ multiple points, then all arrangements in the lattice isomorphism class of $\A$ are also nearly free.
\end{prop}

\begin{proof}
We reproduce here the proof from \cite{DIM_arx}, for completion. Let $\B$ be an arrangement lattice isomorphic to $\A$. Let $H \in \A$ be such that $|\A^H| = b+1$ and $L \in \B$ the line corresponding to $H$ through the lattice isomorphism. $\A \setminus \{H \}$ is free, by Theorem \ref{thm:5.10AD}. Since $\A \setminus \{H \}$ and $\B \setminus \{L \}$ are lattice isomorphic arrangements of cardinal $13$, and the Terao conjecture holds for arrangements of $13$ lines (\cite[Theorem 4.1]{DIM}), if follows that $\B \setminus \{L \}$ is free. Then $\B$ is nearly free, by Theorem \ref{thm:5.10AD}.
\end{proof}

We need to recall two results, stating sufficient, respectively necessary conditions for the combinatoriality  of near freeness.

\begin{prop}(\cite{D2})
\label{prop:Cor1.7D2}
Let $\A$ be nearly free with exponents $(a,b)_{\leq}$. If 
\begin{equation}
\label{eq:Cor1.7D2}
m(\A) \geq a
\end{equation}
then any other line arrangement in the lattice isomorphism class of $\A$ is also nearly free.
\end{prop}

\begin{prop}(\cite{D2})
\label{prop:Prop1.3D2}
Let $\A$ be free or nearly free with exponents $(a,b)_{\leq}$, then 
\begin{equation}
\label{eq:ineqProp1.3D2}
m(\A) \geq \frac{2|\A|}{a+2}
\end{equation}
\end{prop}

Finally, recall that near freeness is proved to be combinatorial in small degrees.

\begin{theorem}(\cite{DIM_arx})
\label{thm:nfree_12lines}
Near freeness is combinatorial for arrangements of up to $12$ lines in the complex projective plane.
\end{theorem}

\begin{lemma}
\label{lemma:rem4.2(iii)}
\begin{enumerate} \item Let $\A$ be nearly free with exponents $(a, b)_{\leq}, \; a+3 \neq b$ and there is a line in $\A$ with either $a+1$ or $b$ multiple points, then any other arrangement in the lattice isomorphism class of $\A$ is also nearly free.
\item Let $\A$ be nearly free with exponents $(a, a+3)$ and there is a line in $\A$ with precisely $a+3$ multiple points, then any other arrangement in the lattice isomorphism class of $\A$ is also nearly free.
\end{enumerate}
\end{lemma}

\begin{proof}
Immediate from \cite[Theorem 5.8]{AD}.
\end{proof}

With this we are ready to prove the nearly free analogue of Corollary \ref{cor:cor5.5A2}.

\begin{theorem}
\label{thm:d_1leq5}
Let $\A$ be a nearly free arrangement of exponents $(a,b)_{\leq}$. If $\{a, b\} \cap \{0,1,2,3,4,5\} \neq \emptyset$, then all the other arrangements in the lattice isomorphism class of $\A$ are also nearly free.
\end{theorem}

\begin{proof}
Say $a \leq 5$.
First of all, by Proposition \ref{prop:Cor1.7D2} we only need to consider the nearly free arrangements $\A$ of exponents $(a, b)$ with $m(\A) < a$. 

If $a < 3$, then it is easy to see that the claim holds. 

$a = 3$ implies that $|\A| = a + b > 4$. But then $m(\A) = 2$, that is, $\A$ is a generic arrangement, which is not nearly free, by Lemma \ref{lemma:generic_nfree}. 

If $a = 4$,  then we are left with an arrangement $\A$ with only double and triple points. Since $\A$ cannot be the generic arrangement, by Lemma \ref{lemma:generic_nfree}, we get $m(\A) = 3$. By Proposition \ref{prop:Prop1.3D2}  we must have that $m(\A) \geq \frac{2|\A|}{a+2}$,
hence $3 \geq \frac{2|\A|}{6}$, i.e. $|\A|  \leq 9$. But for arrangements in this range we already know that near freeness is combinatorial, see Theorem \ref{thm:nfree_12lines}.

If $a=5$, then by Proposition \ref{prop:Cor1.7D2} we only need to look at nearly free arrangements $\A$ with $m(\A) < 5$, and one may as well assume that $m(\A) \in \{3,4\}$.  If $m(\A) = 3$, then, by Proposition \ref{prop:Prop1.3D2}, $|\A| \leq 10$, and this case is covered by Theorem \ref{thm:nfree_12lines}.
If $m(\A) = 4$, then Proposition \ref{prop:Prop1.3D2} implies $|\A| \leq 14$, and we have to discuss separately the situations not included in Theorem \ref{thm:nfree_12lines}. That is, we have to prove that our claim holds for arrangements $\A$ with $|\A|=13$ and exponents $(5,8)$ and for 
arrangements $\A$ with $|\A|=14$ and exponents $(5,9)$. Since $m(\A) = 4$, one readily sees that $\A$ is balanced. In particular, this implies, for $|\A|=14$, that $n_H \geq 5$, and for $|\A|=13$, that $n_H \geq 4$, for any $H \in \A$.

Consider first the case $|\A|=14$ with exponents $(5,9)$. By Proposition \ref{prop:n_H_max}, $n_H \leq 10$, for any $H \in \A$. By Proposition \ref{prop:14nfree}, Lemma \ref{lemma:rem4.2(iii)} and Proposition \ref{prop:n_H_restr}, we can further restrict ourselves to the case when the arrangement $\A$ has $n_H=5$, for all $H \in \A$. Since $\A$ is balanced, we can apply Proposition \ref{prop:exp_dist} to finish the proof in this case.

Let us now consider the case $|\A|=13$ with exponents $(5,8)$. By Proposition \ref{prop:n_H_max}, $n_H \leq 9$, for any $H \in \A$. By Proposition \ref{prop:14nfree}, Lemma \ref{lemma:rem4.2(iii)} and Proposition \ref{prop:n_H_restr}, we can further restrict ourselves to the case when the arrangement $\A$ has $n_H \in \{4, 5, 6\}$.

 Take $\B$ an arrangement lattice isomorphic to $\A$. Then, for any $H \in \B$ one has $n_H \in \{4,5,6\}$. We have to prove that $\B$ is also nearly free.
Assume that there is $H \in \B$ such that $n_H=6$. Let $(e_1,e_2)$ be the exponents of the Ziegler restriction  $(\B^H,m^H)$ of $\B$ onto $H$. Then, it follows from 
 from the incremental computation for the exponents of multiarrangements starting with the exponents for the trivial multiarrangement (see
 \cite[Lemma 4.3]{AN}) that $e_1 \geq n_H - 1 = 5$.
So, $e_1 e_2 \geq 35$. 
 Since $b_2^0(\B)=5 \cdot 7+1=36$, this implies that $\B$ is nearly free for $e_1 e_2=35$ and free for $e_1 e_2=36$. However, no free arrangements can appear in the lattice isomorphism class of the nearly free $\A$, because freeness is combinatorial for $13$ lines arrangements, by \cite{DIM}.

Next assume that $n_H \in \{4,5\}$ for all $H \in \B$. Let $\B'=\B \setminus H$. We already know that 
$\B$ is balanced. Let $\exp(\B^H,m^H)=(e_1,e_2)$. By Theorem \ref{thm: exp_dist_A01},  $e_2-e_1 \le n_H - 2 \leq  3$. Hence $e_1e_2 \ge 5 \cdot 7=35$. Just as before, one has either equality, and then $\B$ is nearly free, or  $e_1e_2 = 36$ and it is free, which is not possible, since this would mean that $\A$ itself should be free, see \cite{DIM}.
\end{proof} 

\section{A characterization of plus-one generated arrangements} 
\label{sec:POG}

\noindent Let $mdr(f^{\A}) = \min \{k \;  | \; AR(f^{\A})_k \neq (0)\}$.  For plus-one generated arrangements $\A$, $mdr(f^{\A})$ coincides to the first exponent of the arrangement, see \cite{DS0}.

\begin{prop}(\cite{AD})
\label{prop:Cor_3.4}
The set of the splitting types onto an arbitrary projective line for the bundle of logarithmic vector fields associated to an arrangement $\A$ is contained in the following: 
$$
\{(r_0, |\A| - 1 - r_0), (r_0 - 1, |\A| - r_0), \dots, (r'_0+1, |\A| - 2 - r'_0), (r'_0, |\A| - 1 - r'_0)\},
$$
where $r_0 = \min(mdr(f^{\A}), [(|\A|-1)/2]), \; r'_0 = \max(mdr(f^{\A}) - \nu(f^{\A}), 0)$.
\end{prop}

 \noindent Here $\nu(f^{\A})$ is an algebraic invariant whose general definition is not essential for the purpose of our paper, suffice to point out that \cite[Proposition 3.7]{DS0} gives a formula for $\nu(f^{\A})$, when $\A$ is a plus-one generated arrangement of exponents $(a, b)$ and level $d$,  in terms of the exponents and level:
$$
\nu(f^{\A}) = d - b + 1.
$$

\begin{prop}
\label{prop:cor3.4AD}
Let $\A$ be a plus-one generated arrangement of exponents $(a, b)$ and level $d$. Then the possible splitting types onto an arbitrary projective line for the bundle of logarithmic vector fields associated to the arrangement are the following: $\{(a, b-1), (a-1, b), \dots, (a-(d-b+1), d)\}$.
\end{prop}

\begin{proof}
In our notations, $mdr(f^{\A})= a$. 
Then the claim of our Proposition follows immediately from the previous Proposition \ref{prop:Cor_3.4}.
\end{proof}

We prove next that the splitting type  $(e_1^L, e_2^L)$ of $\widetilde{AR(f^{\A})}$ over an arbitrary line $L \subset \PP^2$
has actually an even more restrictive set of possible values. Keep in mind that when $L$ is a line in $\A$, the splitting type along $L$ coincides with the exponents of the Ziegler restriction of $\A$ onto $L$, see \cite{Y0}. \\

\noindent Let us recall a basic well known property of the splitting types pairs, to be used next. See for instance \cite[Proposition 3.1]{AD} for a proof.

\begin{prop}
\label{prop:splitting_types_sum}
With the notations above, for any line $L$, one has $e_1^L + e_2^L = |\A| - 1$.
\end{prop}

\begin{theorem}
\label{thm:splitting_type}
Let $\A$ be a plus-one generated arrangement of exponents $(a, b)$ and level $d$, and $L$ an arbitrary line in $\PP^2$. Let $(e_1^L, e_2^L)$ be the splitting type of  $\widetilde{AR(f^{\A})}$ over $L$. Then $(e_1^L, e_2^L) \in \{(a-1, b), (a, b -1), (a+b-d-1, d)\}$.
\end{theorem}

\begin{proof}
 We already know from Proposition \ref{prop:cor3.4AD} that $a - (d -b+1) \leq e_1^L \leq a$. Also, keep in mind that $e_1^L + e_2^L = |\A| - 1 = a+b-1$ (Proposition \ref{prop:splitting_types_sum}).
  If $e_1^L \in \{a-1, a\}$ then the claim of the theorem holds. 

Assume now $e_1^L < a-1$.

Take an arbitrary $L \subset \PP^2$. Consider the graded map $\pi_L$ from Theorem \ref{thm:pi_L}, defined over $\Gamma^*( \widetilde{AR(f^{\A})}) = AR(f^{\A})$. As recalled in Section \ref{sec:preliminaries}, $AR(f^{\A})$ can be identified, via an isomorphism, to a module of derivations associated to $\A, \; D_0(\A)$. Since $\A$ is plus-one generated, this module of derivations is generated by a set of derivations $\{\theta_2, \theta_3, \phi\}, \; \deg(\theta_2) = a, \; \deg(\theta_3) = b, \; \deg(\phi) = d$. 

The codomain of $\pi_L$ is a free module of rank $2$ over $S / (\alpha_L)$.
Choose $\{ \theta_1^L, \theta_2^L \}$ a homogeneous basis of this module.
 Then, due to degree constraints, since $b+1 \leq e_2^L \leq d$,  $\pi(\theta_2), \pi(\theta_3)$ are in the submodule generated by $\theta_1^L$.  Consider the image of the third generator of $D_0(\A)$ through $\pi, \; \pi(\phi) = r\theta_1^L  + s \theta_2^L$, for some $r, s \in S$. Notice that in the latter equality $s \neq 0$, otherwise we would have $\dim_{\k} \Coker (\pi_L) = \infty$ (consider for instance the elements that are not in the image of $\pi_L$, $f \theta_2^L, \; f$ arbitrary monomial), contradiction to the Yoshinaga criterion. Moreover, even when $s \neq 0$,  if $\deg \theta_2^L < d$ (that is, if $\deg(s) > 0$), we can easily construct an infinite series of elements from the codomain not in the image of $\pi_L$ (e.g. elements of type $f \theta_1^L+\theta_2^L$, with $f$ arbitrary monomial), that lead to $\dim_{\k} \Coker (\pi_L) = \infty$. So we are only left with the possibility $ \deg \theta_2^L = d$.
\end{proof}

\begin{corollary}
\label{cor:exp}
 Let $\A$ be a plus-one generated arrangement of exponents $(a, b)$ and level $d$. Let $(\A^H, m^H)$ be the Ziegler restriction with respect to an arbitrary $H \in \A$ and  $(e_1, e_2) := exp(\A^H, m^H)$. Then $(e_1, e_2)\in \left\{(a-1,b),(a,b-1), (a+b-d-1,d)\right\}$.
\end{corollary}

\begin{proof}
Just apply Theorem \ref{thm:splitting_type} to the particular case when $L$ is a line in the arrangement $\A$.
\end{proof}

An analogue of Proposition \ref{prop:n_H_max} holds for plus-one generated arrangements:

\begin{prop}
\label{prop: upper_limit_n_H}
Let $\A$ be a plus-one generated arrangement of exponents $(a, b)$ and level $d$. Then, for any $H \in \A, \; n_H \leq d+1$.
\end{prop}

\begin{proof}
Let $H \in \A$ arbitrary and $(e_1, e_2) := exp(\A^H, m^H)$.
By Proposition \ref{prop:cor3.4AD},  $a - (d -b+1) \leq e_1 \leq a$. 
If for some  $H \in \A$ we would have 
$n_H > d+1$ then $(e_1, e_2) = (|\A| - n_H, n_H -1)$, by Proposition \ref{prop:Yosh_comb}(1). But then $e_1 < a - (d -b+1)$, contradiction.
\end{proof}

\begin{prop}
\label{prop:nfree_old}
Let $\A$ be a plus-one generated arrangement of exponents $(a, b)$ and level $d$ and $H \in \A$.
If $n_H \geq a$, then the only possible values for $n_H$ are $a, a+1, b, b+1, d+1$. 
\end{prop}

\begin{proof}

By Proposition \ref{prop: upper_limit_n_H}, $n_H \leq d+1$.
Next, notice that $b+1 < n_H < d+1$ cannot occur. This follows at once from Corollary \ref{cor:exp} since in this case we would also have  $(e_1, e_2)  =: exp(\A^H, m^H) = (|\A| - n_H, n_H -1)$. 

 It remains to show that $a+1 < n_H < b$ cannot occur. 

\noindent If $n_H \leq  \frac{|\A|+1}{2}$, then we have $e_1 \geq n_H - 1$, as follows incrementally from 
 \cite[Lemma 4.3]{AN}. Then $e_1 > a$, contradiction to Proposition \ref{prop:cor3.4AD}.

\noindent If, on the contrary, $n_H >  \frac{|\A|+1}{2}$, then 
$(e_1, e_2) = (|\A| - n_H, n_H -1)$, by Proposition \ref{prop:Yosh_comb}(1). Since $a+1 < n_H <  b$, we have that $a < n_H-1 < b-1$ and $a < |\A| - n_H < b-1$, so $e_1 > a$, contradiction again to Proposition \ref{prop:cor3.4AD}.\\
\end{proof}

\subsection{A Yoshinaga-type criterion}

Let $\A \ni H$ be an arrangement in $V = \C^3$ and $\{ \theta_1^H, \theta_2^H \}$ a basis for $D(\A^H,m^H)$. We may assume without loss of generality that $\alpha_H = z$. Then taking \text{modulo} $\alpha_H$, for the Ziegler restriction map, sends $S = \C[x,y,z]$ to $\C[x,y]$.\\

Corollary \ref{cor:exp} lists the possible values for the exponents of the Ziegler restrictions for plus-one generated arrangements. Among them, if $b<d$, there is at most one Ziegler restriction with exponents $(a+b-d-1,d)$, as we can see in the following lemma.
 
 \begin{lemma}
 \label{lemma:H_unique}
 Let $\A$ be plus-one generated with exponents $(a,b)$ and level $d$ with $a \leq b <d$. If $H,L \in \A$ are such that $exp(\A^H,m^H)=exp(\A^L, m^L) = (a+b-d-1,d)$, then $H=L$.
 \end{lemma}

\begin{proof} Let  $H \in \A$ be such that $exp(\A^H,m^H) = (a+b-d-1,d)$ and let $\pi﻿=\pi﻿_H: D_H(\A) \rightarrow D(\A^H, m^H)$ be the Ziegler restriction of $\A$ onto $H$. Then $D(\A^H, m^H)$ is generated by a set of derivations  $\{ \theta_1^H,\theta_2^H \}$ with $\deg(\theta_1^H) = a+b-d-1$ and $\deg(\theta_2^H) = d$.
In the notations from Definition \ref{def:POG} with the subsequent remarks, $\{\theta_2, \theta_3,  \phi \}$ is a minimal set of generators for $D_H(\A), \; \deg(\theta_2) = a, \; \deg(\theta_3) = b, \; \deg(\phi) = d$. Since $a \leq b <d, \; \pi(\theta_2), \pi(\theta_3)$ are in the submodule of $D( \A^H, m^H)$ generated by $\theta_1^H$. Let $\pi(\theta_2) = f \theta_1^H$ and $\pi(\theta_3) = g \theta_1^H, \; f,g \in S$. Then $\pi(g \theta_2 - f \theta_1) \in \Ker(\pi)$, hence the homogeneous relation of degree $d+1$
$$
g \theta_2 - f \theta_3 - \alpha_H \theta =0,
$$
for some derivation $\theta \in D_H(\A)$.
Assume $L \in \A$ is such that $exp(\A^L,m^L) = (a+b-d-1,d)$.  As before we get a homogeneous relation of degree $d+1$
$$
g' \theta_2 - f' \theta_3 - \alpha_L \theta' =0,
$$
for some derivation $\theta' \in D_H(\A)$.
Since $b<d$, in the unique generating syzygy from Definition \ref{def:POG}, in the notations established there, $\alpha \neq 0$ and both $\alpha$ and $\phi$ are uniquely determined (modulo multiplication by an element in $\k$). Just express $\theta$ and $\theta'$ in the previous relations in terms of the generators of $D_H(\A)$ and recall that both relations are homogeneous of degree $d+1$ to arrive at the conclusion that $\alpha_L$ and $\alpha_H$ are linearly dependent, hence $H=L$.
\end{proof}
 
 \begin{corollary}
 \label{cor:maxline}
  If $\A$ is plus-one generated with exponents $(a,b)$ and level $d$ with $a \leq b <d$, then there is at most one line $H \in \A$ such that $n_H = d+1$. 
 \end{corollary}
 
 \begin{proof}
 Immediate from Lemma \ref{lemma:H_unique} and Proposition \ref{prop:Yosh_comb}(1).
 \end{proof}
 
 \begin{remark}
 \label{rem: nfree_parallel}
 Corollary \ref{cor:maxline} implies an extension of Proposition \ref{prop:n_H_max} (case $b=d$) to the plus-one generated case $b<d$:  it immediately follows from the corollary that there exists $H \in \A$ such that  $n_H < d+1$, as soon as $|\A| \geq 2$. The claim itself of the corollary does not hold however for plus-one generated arrangements with $b=d$, see the next counter-example.
 \end{remark}
 
 \begin{example}
 \label{ex:counter}
 Let $\A \subset V = \C^3$ be defined by  $f_{\A} = xyz(x+4y)(x+5y+z)(y+z)$. Then $\A$ is plus-one generated with exponents $(3,3)$ and level $3$. For $H:= \{x+4y=0\}$, respectively $L:= \{x+5y+z=0\}, \; n_H = n_L =4$.
 \end{example}

\begin{remark} The above corollary is also obtained as a consequence of addition-theorem for plus-one generated arrangements, see \cite[Theorem 5.5]{A0} . Namely, if $\A$ is plus-one generated with exponents $(a,b)$ and level $d$, with $a \leq b <d$ and $H\in \A$ such that $n_H=d+1$, then by \cite[Theorem 5.5]{A0}, $\A' = \A-\left\{H\right\}$ is free with exponents $(a-1,b-1)$. If $L\in \A$, $L\neq H$, then $L$ is in the free arrangement $\A'$ and thus, by Theorem \ref{thm:cor1.2A2}, $|\A'^L| \leq a<d$ or $|\A'^L|=b<d$.  In any case, $n_L \leq d$.
\end{remark}

Hence, for a plus-one generated arrangement in a 3-dimensional space with exponents $(a,b)$ and level $d$, with $a\leq b <d$, there exists at most one line $H\in \A$ such that $exp(\A^H, m^H)=(a+b-d-1,d)$. In Example \ref{ex:POGoneH}, $\A$ is a plus-one generated arrangement in which such a line $H$ exists. The arrangement $\A$ from Example \ref{ex:POGnoH} is a plus-one generated arrangement without any such line $H$.

\begin{example}\label{ex:POGoneH} Let $V=\C^3$ with coordinates $x,y,z$ and let $\A$ be a plus-one generated arrangement with exponents $(3,3)$ and level $4$ defined by $f_{\A}=xyz(x+y)(-x+2y+z)(x+2y+z)$. In this case, there is $H:=\left\{z=0\right\}$ for which Ziegler's multirestriction $\left(\A^H,m^H\right)$ of $\A$ onto $H$, defined by $xy(x+y)(-x+2y)(x+2y)$, has the exponents $(a+b-d-1,d)=(1,4)$.
\end{example}

\begin{example}\label{ex:POGnoH} Let $V=\C^3$ with coordinates $x,y,z$ and let now $\A$ be a plus-one generated arrangement with exponents $(3,3)$ and level $4$  defined by $f_{\A}=xyz(x+y)(-x+2y+z)(x-2y+z)$. One can check that, for any $H\in \A$, Ziegler's multirestriction $\left(\A^H,m^H\right)$ has exponents $(2,3)$. Hence, there is no line with exponents $(1,4)$.
\end{example}

The next example shows that, in Lemma \ref{lemma:H_unique}, the condition $b<d$ is essential. Namely, we give an example of a plus-one generated arrangement with exponents $(a,b)$ and level $d=b$ (i.e. a nearly free arrangement of exponents $(a,b)$) where any Ziegler's multirestriction has the exponents $(a+b-d-1,d)$.

\begin{example}
\label{ex:generic_4lines}
 Let $V=\C^3$ with coordinates $x,y,z$ and let $\A$ be an arrangement defined by $f_{\A}=xyz(x+y+z)$, i.e. a generic arrangement of $4$ lines. Then $\A$ is plus-one generated with exponents $(2,2)$ and level $2$. One can check that, for any $H\in \A$, its Ziegler multirestriction $\left(\A^H,m^H\right)$ has exponents $(a+b-d-1,d)=(1,2)$. Hence, all $4$ lines have Ziegler's multirestriction exponents $(a+b-d-1,d)$. 
\end{example}
 
\subsection{Proof of Theorem \ref{thm:POG_characterisation}}

Assume $\A$ is an arrangement with property {\bf [P]}(1): there exists $ \alpha \in V^*$ such that $\theta_1^H, \alpha \theta_2^H \in \im(\pi)$.  That is, there are $\theta_1, \theta_2 \in D_H(\A)$ such that $\pi(\theta_1) = \theta_1^H$ and $\pi(\theta_2) =\alpha  \theta_2^H$.
Say $\alpha = ax +by$, and at least one of $a,b \in \k$ is non-zero, so we may as well assume that $a \neq 0$. 

We know from Theorem \ref{thm:Yoshinaga} that $\dim_{\k} \Coker (\pi) < \infty$. 
We will describe a set of generators of $\Coker (\pi)$ over $\k$. First take an arbitrary element in $D(\A^H,m^H)$, $\phi = f \theta_1^H+ g \theta_2^H, \; f, g \in \k[x,y]$. 
Then $f \theta_1^H = \pi(f \theta_1)$, so $\phi \equiv g \theta_2^H  \;(\text{modulo}\;  \im(\pi))$.
Notice that, since $(ax +by) \theta_2^H \in \im(\pi)$, either 
$$
x \theta_2^H  \in \im(\pi) \; \text{if}\;  b = 0 \;  \text{or},  
$$

$$
x \theta_2^H \equiv (-b/a) y \theta_2^H \;(\text{modulo} \; \im(\pi)) , \; \text{if} \; b \neq 0.
$$
Hence, there exists $g_y \in \k[y]$ such that  $\phi \equiv g_y \theta_2^H  \;(\text{modulo}\;  \im(\pi))$. 
This shows that the generators over $\k$ of $\Coker (\pi)$ can be taken of type $h \theta_2^H$, with $h \in \k[y]$. Take such a set of generators 
$$
G = \{f_1 \theta_2^H, \dots, f_{q} \theta_2^H\}, \; f_i \in \k[y],
$$

\noindent $\{f_1, \dots, f_q\}$ of degrees $\{m_1 \leq  \dots \leq m_q\}$.
  
  For an arbitrary $n \in \N$, consider the element $y^n  \theta_2^H \in D(\A^H, m^H)$. Then 
  $$
  y^n \theta_2^H  \equiv \sum_{i}a_i f_i \theta_2^H \;(\text{\text{modulo}}\;  \im(\pi)), \; a_i \in \k.
  $$
  
 If $n> m_q$ then $y^n \theta_2^H  \in  \im(\pi)$. 
   Let $m$ be the minimal natural number with the property that $y^m \theta_2^H  \in \im(\pi)$. It is clear that $m \geq 1$. By an eventual further reduction $ (\text{\text{modulo}} \; \im(\pi))$ we can assume that the polynomials $f_i \in \k[y]$ that describe the set $G$ only contain monomials $y^j$ such that $y^j \theta_2^H  \notin \im(\pi)$. 
   Finally, observe that $\{ y^i \theta_2^H \; | \; 0 \leq i \leq m-1 \}$ is a complete set of generators over $\k$ of $\Coker (\pi)$.
    
   It is not hard to see that $\{ \theta_1^H, \alpha \theta_2^H, y^m \theta_2^H \}$ generate $\im(\pi)$ over $k[x,y]$.
    If $f \theta_1^H + g \theta_2^H \in \im(\pi)$, then $g \theta_2^H \in \im(\pi)$ and, as before, this reduces in two steps, first to $g_y \theta_2^H \in \im(\pi)$ and finally to $g_y ^{<m}\theta_2^H \in \im(\pi)$, where $g_y ^{<m} = \sum_{i<m} b_i y^i$ is the truncation of $g_y$ to its part of degree $<m$. We need to prove that $g_y ^{<m}=0$. 
   If $\sum_i b_i y^i$ would be non-zero, then consider $i_0$ the minimal $i$ such that $b_{i_0} \neq 0$ and multiply by $y^{(m-i_0-1)}$ the element $\sum_{i<m} b_i y^i \theta_2^H \in \im(\pi)$. It follows that $y^{m-1} \theta_2^H \in \im(\pi)$, contradiction to the minimality of $m$. In conclusion, indeed the set $\{ \theta_1^H, \alpha \theta_2^H, y^m \theta_2^H \}$ generates $\im(\pi)$ over $k[x,y]$.
   
    Let $\psi \in D_H(\A)$ such that $\pi(\psi) = y^m  \theta_2^H$. We want to prove that $\{ \theta_1, \theta_2, \psi \}$ is a minimal set of generators for $D_H(\A)$ as $\k[x,y,z]$ module. 
    Let $\theta \in D_H(\A)$. Then 
$$
\pi(\theta) = f \theta_1^H + g \alpha \theta_2^H + h y^m \theta_2^H = \pi(f \theta_1 + g \theta_2 + h \psi), \; f,g,h \in \k[x,y] \subset \k[x,y,z].
$$ 
Hence $\theta - f \theta_1 - g \theta_2 - h \psi \in \Ker(\pi) = z \cdot D_H(\A)$, that is 
$$
\theta = f \theta_1 + g \theta_2 + h \psi + z \theta_0, \; \theta_0 \in D_H(\A), \; \deg(\theta_0) < \deg(\theta).
$$
 We then apply $\pi$ to $\theta_0$ and repeat the above procedure. Since the degrees strictly decrease, after a finite number of steps we will express $\theta$ in terms of the generators $\{ \theta_1, \theta_2, \psi \}$.
    
The set of generators $\{ \theta_1, \theta_2, \psi \}$ for $D_H(\A)$ is minimal, since $\A$ is not free. 
We have a relation among these three generators, derived from $\pi(y^m \theta_2 - \alpha \psi)=0$.
This is equivalent to $y^m \theta_2 - \alpha \psi = z \cdot \gamma$, for some $\gamma \in D_H(\A)$. But $\gamma = f \theta_1 + g \theta_2 + h \psi$ for some $f,g,h \in \k[x,y,z]$, so $y^m \theta_2 - \alpha \psi = z \cdot (f \theta_1 + g \theta_2 + h \psi)$,  or $z f \theta_1 + (z g - y^m) \theta_2 + (z h - \alpha) \psi = 0$, with $\deg(h) = 0$. This relation among the generators of $D_H(\A)$ is unique by \cite[Cor. 2.2]{DS0}.
   
 The case when {\bf[P]}(2) holds can be treated similarly.\\
 
 Conversely, let $\A$ be a plus-one generated arrangement with exponents $(a, b)$ and level $d$. Then, for an arbitrary $H \in \A, \; exp(\A^H, m^H) \in \{(a, b-1), (a-1, b), (a+b-d-1,d)\}$. By Lemma \ref{lemma:H_unique}, when $b<d$, there is at most one $H \in \A$ such that $exp(\A^H, m^H) = (a+b-d-1, d)$. Then, if $|\A| \geq 2$ we may safely assume that there exists $H \in \A $ such that $exp(\A^H, m^H) \in \{ (a, b-1), (a-1, b) \}$. We may assume $\alpha_H = z$.  
 
 Consider $\{ \theta_1, \theta_2, \psi \}$ the minimal generating set of derivations in $D_H(\A)$ of degrees $a, b$, respectively $d$. As before, denote by $\pi$ the Ziegler map onto $H$. It is clear that $\{ \theta_1, \theta_2 \} \cap Ker(\pi) = \emptyset$.
 
 Assume $exp(\A^H, m^H) = (a, b-1)$, then let $\theta_1^H = \pi(\theta_1)$. Consider $\pi(\theta_2) = f \theta_1^H + g \theta_2^H, \; f,g \in \k[x,y], \;\deg(g)=1$.  Then  $g \theta_2^H \in Im(\pi)$, which closes this case, unless $g=0$. If $g = 0$ and $b < d$, this leads to a dependency relation among $\theta_1$ and $\theta_2$, since $g=0$ implies $\pi(\theta_2) = t \theta_1^H, \; t \in \k[x,y]$. Then $\pi(\theta_2 - t \theta_1) = 0$, i.e. $\theta_2 - t \theta_1 = z \delta$, with $\delta \in D_H(\A)$ of degree $b-1$. This implies $\delta = q  \theta_1$ for some $q \in \k[x,y,z]$, so $\theta_2 - (t+z q) \theta_1=0$, contradiction.
 If $g = 0$ and $b = d$ consider the image of $\psi$ through $\pi, \; \pi(\psi) =   f' \theta_1^H + g' \theta_2^H, \; f',g' \in \k[x,y], \; \deg(g') = 1$. Then necessarily $g' \neq 0$, since $\dim_{\k}\Coker(\pi) < \infty$, and $g' \theta_2^H \in Im(\pi)$.

 If $exp(\A^H, m^H) = (a-1, b)$, then $\theta_1^H \notin Im(\pi)$, for any basis $\{\theta_1^H,\theta_2^H  \}$ of $D(\A^H, m^H)$. For such arbitrary basis, $\pi(\theta_1) = f \theta_1^H + g \theta_2^H, \;  f,g \in \k[x,y], \; deg(f)=1$. 
 
 Moreover, if $a< b$, then $\pi(\theta_1) = f \theta_1^H, \; f  \in \k[x,y], \; \deg(f) = 1$ and $\pi(\theta_2) = r \theta_1^H + s  \theta_2^H, \; r,s \in \k[x,y], \; \deg(s)=0$. If $s \neq 0$, we consider the new  basis $\{ \theta_1^H, r \theta_1^H + s \theta_2^H \}$ in $D(\A^H, m^H) $, which satisfies the requirements of the theorem.
 If $s = 0$ and  $b < d$ then we get again, as before, a dependency relation among $\theta_1$ and $\theta_2$, contradiction. Otherwise, let $s = 0$ and  $b = d$. Let $\pi(\psi) =   r' \theta_1^H + s' \theta_2^H, \; r', s' \in \k[x,y], \; \deg(s') = 0$. Since $\dim_{\k}\Coker(\pi) < \infty$, it follows that $s' \neq 0 $ and we consider the new basis $\{ \theta_1^H,  r' \theta_1^H + s' \theta_2^H \}$ in $D(\A^H, m^H) $, for which property [P] holds.
 
 If, on the other hand,  $a = b$, then, for an arbitrary basis $\{\theta_1^H,\theta_2^H  \}$ of $D(\A^H, m^H)$,  $\pi(\theta_1) = f \theta_1^H + q  \theta_2^H, \; f, q \in \k[x,y], \; \deg(f)=1, \; \deg(q) = 0$ and $\pi(\theta_2) = r \theta_1^H + s  \theta_2^H, \; r, s \in \k[x,y], \; \deg(r)=1, \; \deg(s)=0$. Notice that $f, r$ cannot be simultaneously $0$. Otherwise $s \theta_1 - q \theta_2 = z \theta_0, \; \theta_0 \in D_H(\A)$ with $\deg(\theta_0) < a$, contradiction. So we can assume without loss of generality that $f \neq 0$.
 
 Assume $q, s$ are simultaneously $0$. Consider the image of $\psi$ through $\pi, \; \pi(\psi) =   f' \theta_1^H + g' \theta_2^H, \; \deg(f') = 1, \; \deg(g') = 0$. Then necessarily $g' \neq 0$, since $\dim_{\k}\Coker(\pi) < \infty$. In this case, take the basis $\{ \theta_1^H, f' \theta_1^H + g' \theta_2^H \}$ in $D(\A^H, m^H)$, for which property [P] holds.

 So, for the remainder of this proof, we can assume that $q, s$ are not simultaneously $0$.
If $q=0$, then $s \neq 0$ and $f \theta_1^H \in \im(\pi)$ and consider the new  basis $\{\theta_1^H, r \theta_1^H + s \theta_2^H\}$ in $D(\A^H, m^H) $, which satisfies the requirements of the theorem. If $q \neq 0, \; (s f - q r)\theta_1^H \in \im(\pi)$ and we take the basis $\{ \theta_1^H, f \theta_1^H + q \theta_2^H \}$ in $D(\A^H, m^H)$. To complete the argument in this case, it remains to see that $s f - q r \neq 0$. If $sf = qr$, then $\pi(s \theta_1) = \pi( q \theta_2)$ (incidentally this implies $s \neq 0$) and then just as before $s \theta_1 - q \theta_2 = z \theta_0, \; \theta_0 \in D_H(\A)$ with $\deg(\theta_0) < a$, contradiction.
\qed

\begin{example}\label{ex:necsufcond} Let $\A$ be the arrangement defined by $f_{\A}=xyz(x+z)(-x+y+2z)(x+y+2z)$. The Ziegler multirestriction of $\A$ onto $H:=\left\{z=0\right\}$ is defined by $f_{\A^H}=-x^2y(x^2-y^2)$. Let $\left\{\delta_1^H,\delta_2^H\right\}$ be the basis in $D(\A^H, m^H)$ defined by $\delta_1^H=x^2\partial_x+xy\partial_y$ and $\delta_2^H=2x^2y\partial_x+y(x^2+2xy-y^2)\partial_y$. Then, $\exp(\A^H, m^H)=(2,3)$. 

Let $\{ \theta_1, \theta_2 \} \subset D_H(\A)$ defined by $\theta_1=y(x+z)\left(x\partial_x+(y+2z)\partial_y\right)$ and $\theta_2=2xy(x+z)\partial_x+y(x^2+2xy-y^2+4xz-2yz)\partial_y$.

Then, $\pi_H(\theta_1)=y\delta_1^H$ and $\pi_H(\theta_2)=\delta_2^H$. Hence, there exists the linear form $\alpha=y\in S$ such that $\left\{\alpha\delta_1^H, \delta_2^H\right\}\subset Im(\pi)$. Thus the Property [P] from Theorem \ref{thm:POG_characterisation} is satisfied and the arrangement $\A$ is plus-one generated with exponents $(3,3)$ and level $4$.
\end{example}

\begin{example}\label{ex2:necsufcond}
 Let $\A$ be the arrangement defined by $f_{\A}=xyz(x-y)(x+y)(y+z)(x+4y+z)$. The Ziegler multirestriction of $\A$ onto $H:=\left\{z=0\right\}$ is defined by $f_{\A^H}=xy^2(x^2-y^2)(x+4y)$. Let $\left\{\delta_1^H,\delta_2^H\right\}$ be the basis in $D(\A^H, m^H)$ defined by $\delta_1^H=xy\partial_x+y^2\partial_y$ and $\delta_2^H=x(x+4y)(x^2-7xy-12y^2)\partial_x-y^2\left(x+4y\right)\left(7x+11y\right)\partial_y$. Then $\exp(\A^H, m^H) = (2,4)$. 

Let $\{ \theta_1, \theta_2 \} \subset D_H(\A)$ defined by $\theta_1=(y+z)(x+4y+z)\left(x\partial_x+y\partial_y\right)$, $\theta_2=x(x+4y+z)(x^2-7xy-12y^2-7xz-11yz)\partial_x-y(y+z)(x+4y+z)(7x+11y)\partial_y$ and $\psi=-xz(y+z)(4xy-12y^2+xz-11yz-8z^2)\partial_x+yz(y+z)(x^2-4xy+11y^2-xz+11yz+8z^2)\partial_y$. 

Then, $\pi_H(\theta_1)=\alpha\delta_1^H$ and $\pi_H(\theta_2)=\delta_2^H$, where $\alpha$ is the linear form $x+4y\in S$. Since $\left\{\alpha\delta_1^H, \delta_2^H\right\}\subset Im(\pi)$, the Property [P] from Theorem \ref{thm:POG_characterisation} is satisfied and the arrangement $\A$ is plus-one generated of exponents $(3,4)$ and level $5$.
\end{example}

\end{document}